\documentclass{article}
\usepackage{amssymb}
\title{ On the  Topology induced by Probabilistic Modular spaces   \\[0.3cm]}

\author{{ Kamal Fallahi   , \,\,\,Kourosh Nourouzi
 }\\[0.4cm]
{ \em  Department of Mathematics,  K. N. Toosi University of
 Technology,}\\
{\em P.O. Box 16315-1618, Tehran, Iran.}\\}

\newenvironment{proof}{\noindent {\em {Proof}}.}{$\square$
\medskip}

\newtheorem{definition}{Definition}

\newtheorem{theorem}{Theorem}

\newtheorem{remark}{Remark}
\newtheorem{lm}{Lemma}

\begin{document}

\maketitle \begin{abstract} In this  note, we   investigate some topological properties of probabilistic modular spaces.
\end{abstract}

\section{Introduction}

A modular on a real linear space $X$ is a real functional $\rho$
on $X$ satisfying the following conditions:
\begin{description}
\item \ 1. $\rho(x)=0$ iff $x=0$; \item \ 2. $\rho(x)=\rho(-x)$, for all $x\in X$;
\item \ 3. $\rho(\alpha x+\beta y)\leq \rho(x)+\rho(y)$, for all
$x,y\in X$ and $\alpha, \beta \geq 0$, $\alpha +\beta=1$.
\end{description}
The vector subspace $X_\rho = \{x\in X: \rho(ax)\rightarrow 0
\,\,\,\rm{as} \,\,\,a\rightarrow0\}$ of $X$ is called a modular
space (see e.g.,   \cite{koz, mus, k3}).

 According to the  notions  modular and probabilistic norm
(see e.g.,\cite{agh, eng, meng}), the concept of a  probabilistic modular  first initiated in \cite{k4} and then followed in \cite{k1, k2, k5}.

 In this  paper,   some topological properties of probabilistic modular spaces are investigated.

 We will denote the set of all non-decreasing functions $f: \mathbb{R}
 \rightarrow {\mathbb{R}_0}^+$ satisfying
 $\inf_{t\in \mathbb{R}}f(t)=0$, and
$\sup_{t\in \mathbb{R}}f(t)=1$ by $\Delta$.

\begin{definition}{\rm A pair $(X, \mu)$ is called
a  probabilistic modular space (briefly, $\mathcal{PM}$-space)  if
$X$ is a real vector space, $\mu$ is a mapping from $X$ into
$\Delta$ (for $x\in X$, the function $\mu(x)$ is denoted by
$\mu_x$) satisfying the following
conditions:
\begin{description}
\item \ 1.$\mu_x(0)=0$; \item \ 2. $\mu_x(t)=1$ for all $t>0$ iff
$x=0;$ \item \ 3. $\mu_{-x}(t)=\mu_x(t)$; \item \ 4. $\mu_{\alpha
x+\beta y}(s+t)\geq \mu_x(s)\wedge \mu_y(t)$ for all $x, y \in X$,
and $\alpha, \beta, s,t \in {\mathbb{R}_0}^+$, $\alpha + \beta
=1.$
\end{description}}
\end{definition}

A $PM$-space $(X, \mu)$ is said to satisfy $\Delta_2$-condition if
there is a constant $c> 0$ such that $\mu_{2x}(t) \geq \mu_x
(\frac{t}{c})$ for any $x \in X$ and $t>0.$
\begin{definition} For a positive number $\beta\leq 1$, the modular $\mu$ of  the
  $PM$-space $(X,\mu)$  is  said to
be $\beta$-homogeneous if it satisfies the equality
$$\mu_{\alpha x} (t) = \mu_x(\frac{t}{|\alpha|^\beta})$$ for every
$x\in X$, $t>0$ and $\alpha \in\mathbb{R}$.
\end{definition}

%
\begin{definition} {\rm Let $(X,\mu)$ be a $\mathcal{PM}$-space.
For $x\in X$, $t>0$, and $0<\alpha<1$, the $\mu$-ball centered at
$x$ with radius $r$ is defined by $$ B(x,r,t)=\{y\in X:
\mu_{x-y}(t)>1-\alpha \}.$$}
\end{definition}

\section{Results}
\begin{lm} \label{1} Let $(X,\mu)$ be a $\mathcal{PM}$-space. If  $y\in B(x,\alpha, t)$, then
there exists $t^\ast \in (0,t)$ such that
$\mu_{x-y}(t^\ast)>1-\alpha$.
\end{lm}
\begin{proof} Suppose on the contrary that  for all $s\in (0,t)$,
 $\mu_{x-y}(s)\leq 1-\alpha$. By the left continuity of
$\mu_x(\cdot)$ at $t$, we get $\mu_{x-y}(t)\leq 1-\alpha$, which
is a contradiction.
\end{proof}
\begin{theorem} \label{base} Let $(X, \mu)$ be a $\mathcal{PM}$-space, where $\mu$ satisfies  $\Delta_2$-condition. Then the probabilistic modular $\mu$ induces  a topology
$\tau_\mu$ on $X$ with a basis $$\Gamma= \ \{B(x, \alpha, t): x\in X, \alpha\in (0,1), t>0\}.$$
\end{theorem}
\begin{proof} It is evident that for every $x\in X$, there exist
 $B\in \Gamma$ containing  $x$.
We  show that if $z\in B(x,\alpha,t)\in \Gamma$, then there
exists $B(z,\alpha',t')\in \Gamma$
 such that $B(z,\alpha', t')\subset B(x, \alpha, t)$.
 If $z\in B(x, \alpha, t)$, then $\mu_{x-z}(t)>1-\alpha$ and
 therefore  by Lemma \ref{1}, there exists $t^\ast \in (0,ct)$ such that
  $\alpha^\ast=\mu_{x-z}(\frac{t^\ast}{c})>1-\alpha$,
 where $c$ is the $\Delta_2$-constant. Choose $s, \alpha_1\in (0,1)$ such
 that $\alpha^\ast \wedge \alpha_1>1-s> 1-\alpha$. If $y\in B(z,1-\alpha_1, \frac{t-t^\ast}{c})$ ,
 then $\mu_{y-z}(\frac{t-t^\ast}{c})> \alpha_1$. We have
 \begin{center}
 \begin{tabular}{lllllllll}
 $\mu_{x-y}(t)$&$=$&$ \mu_{(x-z)+(z-y)}(t-t^\ast+t^\ast)$\\
 &$\geq$&$\mu_{2(x-z)}(t-t^\ast)\wedge\mu_{2(z-y)}(t^\ast)$\\&$\geq$&
 $\mu_{x-z}(\frac{t-t^\ast}{c})\wedge \mu_{z-y}(\frac{t^\ast}{c})$\\
 &$\geq$&$\alpha^\ast \wedge \alpha_1$\\&$>$&$1-s$\\&$>$&$1-\alpha$.
 \end{tabular}
 \end{center}
 That is $y\in B(x, \alpha, t)$ and hence  $B(z,1-\alpha_1, \frac{t-t^\ast}{c})\subset B(x, \alpha, t)$.
 Now let $y\in B_1\cap B_2$, where  $B_1, B_2\in \Gamma$.  Then there
 exist $\alpha_1, \alpha_2\in (0,1)$ and $t_1, t_2>0$ such that $$B(y, \alpha_1, t_1)\subset
 B_1,\,\,\,\,\,\,\,
 B(y, \alpha_2, t_2)\subset B_2.$$
 Putting  $t'=\min\{t_1, t_2\}$ and $\alpha'=\min\{\alpha_1,
 \alpha_2\}$, we clearly obtain
  $B(y, \alpha', t')\subset B_1 \cap B_2$.
\end{proof}
\begin{remark}  The set $\{B(x,\frac{1}{n}, \frac{1}{n}): \ \ n \in \mathbb{N}\}$
is a locally base for $x$ of  $\mathcal{PM}$-space $(X,\mu)$.
\end{remark}
\begin{proof} Let $B(x, \alpha, t)$ be a typical  element of the base given in Theorem \ref{base}.
There exist natural number $n$ such that $\frac{1}{n}<t\wedge
\alpha$. Now if $z\in B(x,\frac{1}{n},\frac{1}{n})$, we have
$\mu_{x-z}(\frac{1}{n})>1-\frac{1}{n}$, thus
$\mu_{x-z}(t)>\mu_{x-z}(\frac{1}{n})>1-\frac{1}{n}>1-\alpha.$
\end{proof}

The preceding theorem implies that the induced  topology
$\tau_\mu$ of  $\mathcal{PM}$-space $(X,\mu)$ is first-countable.
\begin{theorem}
 Let $(X, \mu)$ be a $\mathcal{PM}$-space, where $\mu$  satisfies   $\Delta_2$-condition. Then $(X, \mu)$ is Housdorff.
\end{theorem}
\begin{proof} Let $x$ and $y$ be two distinct points of $X$. Then there exists $t_0>0$ such that $\mu_{x-y}(t_0)<1$.
 Choosing $\alpha_1$ with $\mu_{x-y}(t_0)<\alpha_1<1$, we have
  $$B(x, 1-\alpha_1, \frac{t_0}{2c})\cap B(y, 1-\alpha_1, \frac{t_0}{2c})=\emptyset .$$
 Otherwise, if $z\in B(x, 1-\alpha_1, \frac{t_0}{2c})\cap B(y, 1-\alpha_1, \frac{t_0}{2c})$,
  then  \begin{center}
 \begin{tabular}{lllllllll}$\mu_{x-y}(t_0)$&$=$&$\!\!\!\!\!\!\!\!\!\!\!\!\!\!\!\mu_{(x-z)+(z-y)}(\frac{t_0}{2}+\frac{t_0}{2})$\\
 &$\geq$&$\!\!\!\!\!\!\!\!\!\!\!\!\!\!\!\mu_{x-z}(\frac{t_0}{2c})\wedge\mu_{z-y}(\frac{t_0}{2c})$\\&$>\alpha_1,$
 \end{tabular}
 \end{center}
 which is a  contradiction.
\end{proof}
\begin{lm}\label{11} Let $(X,\mu)$ be a $\beta$-homogeneous $\mathcal{PM}$-space, $\alpha\in(0,1)$ and $t>0$.
 Then \begin{description}
\item \ (a)\ $B(x, \alpha, t)= x+ B(0, \alpha, t)$. \item \ (b)\
$B(0, \alpha, t^\beta)=t B(0, \alpha, 1)$. \item \ (c)\ If \
$t_1\leq t_2$, then $B(0, \alpha, t_1)\subset B(0, \alpha,t_2)$.
\item \ (d)\ If \ $\alpha_1\leq \alpha_2$,  then $B(0, \alpha_1,
t)\subset B(0, \alpha_2, t)$.
\end{description}
\end{lm}
\begin{proof}
\begin{description}
\item \ (a)
\begin{center}
 \begin{tabular}{lllllllll} $x+B(0, \alpha, t)$&$=$&$x+\{y\in X: \mu_y(t)>1-\alpha\}$\\
 &$=$&$\{x+y: \mu_y(t)>1-\alpha\}$\\&$=$&$\{z\in X: \mu_{z-x}(t)>1-\alpha\}$\\
 &$=$&$B(x, \alpha, t).$\end{tabular}
 \end{center}
 \item \ (b)\ \begin{center}
 \begin{tabular}{lllllllll}  $tB(0,\alpha, 1)$&$=$&$t\{y\in X: \mu_y(1)>1-\alpha\}$\\
 &$=$&$\{ty: \mu_y(1)>1-\alpha\}$\\&$=$&$\{x\in X: \mu_{\frac{x}{t}}(1)>1-\alpha\}$\\
 &$=$&$\{x\in X: \mu_x(t^\beta)>1-\alpha\}$\\
 &$=$&$B(x, \alpha, t^\beta).$\end{tabular}
 \end{center}
 \item \ (c)\ If $t_1\leq t_2$ and $x\in B(0, \alpha, t_1)$, then $\mu_x(t_2)\geq \mu_x(t_1)>1-\alpha$, i.e.,
  $x\in B(0, \alpha, t_2)$.
 \item \ (d)\ If $\alpha_1\leq \alpha_2$ and $x\in B(0, \alpha_1, t)$, then $\mu_x(t)>1-\alpha_1>1-\alpha_2$, i.e., $x\in B(0, \alpha_2, t)$. \end{description}
\end{proof}


The following condition of a modular $\mu$ will be needed
later.\\

 $(\mathbf{\Upsilon})$ For any non-zero element $x\in X$,
$\mu_x(\cdot)$ is a continuous function on $\mathbb{R}$ and
strictly increasing on $\{t: 0<\mu_x(t)<1\}$.

\begin{theorem} Let  $(X,\mu)$ be a $\beta$-homogeneous $\mathcal{PM}$-space  satisfying $(\mathbf{\Upsilon})$.
Then $(X,\mu)$ is a Hausdorff topological vector space, whose
 local base of origin $0$ is  $\frak{S}=\{B(0, \alpha, t): \
t>0, \alpha\in (0,1)\}$.
\end{theorem}
\begin{proof} To show that $\frak{S}$ is a base, let $B_1=B(0, \alpha_1, t_1)$ and $B_2=B(0, \alpha_2, t_2)$.
Choosing $\alpha_0=\min\{\alpha_1, \alpha_2\}$ and $t_0=\min\{t_1,
t_2\}$, then $B_0=B(0, \alpha_0, t_0)\subset
B_1\cap B_2$. Therefore the set $\{B(0, \alpha, t): \ t>0,
\alpha\in (0,1)\}$ is a base of
 the topology on $X$. Now, the space  $(X,\mu)$ is a Hausdorff space. In fact, if
$x\in X$ and $x\neq 0$,   then there exist $\alpha_0\in (0,1)$ ,
$t_0 \in \{t: 0<\mu_x(t)<1\}$,
 such that $\mu_x(t_0)<1-\alpha_0$. Thus $x\notin B(0, \alpha_0, t_0)$ and
$$B(0, \alpha_0, \frac{t_0}{2^{\beta+1}})\cap B(x, \alpha_0, \frac{t_0}{2^{\beta+1}})=\emptyset .$$
 Otherwise, if $$y\in B(0, \alpha_0, \frac{t_0}{2^{\beta+1}})\cap B(x, \alpha_0, \frac{t_0}{2^{\beta+1}}),$$
  then $$\mu_y(\frac{t_0}{2^{\beta+1}})>1-\alpha_0, \,\,\,\,\,\,\,\, \mu_{x-y}(\frac{t_0}{2^{\beta+1}})>1-\alpha_0$$ and we would have
     \begin{center}
 \begin{tabular}{lllllllll} $\mu_x(t_0)$&$\geq$&$\mu_{2(x-y)}(\frac{t_0}{2})\wedge\mu_{2y}(\frac{t_0}{2})$\\
&$=$&$\mu_{x-y}(\frac{t_0}{2^{\beta+1}})\wedge\mu_y(\frac{t_0}{2^{\beta+1}})$\\&$>$&$1-\alpha_0,$\end{tabular}
 \end{center}
which is a  contradiction.\\
 To see that the vector space operations (addition and scalar multiplication)
are continuous with respect to the topology induced by $\mu$, it
suffices to show that  if $B=B(0, \alpha, t)$, then there exist
$B_1=B(0, \alpha_1, t_1)$ and $B_2=B(0, \alpha_2, t_2)$ such that
$B_1+ B_2\subset B$. Choose  $\alpha_1, \alpha_2< \alpha$ and
$t_1, t_2<\frac{t}{2^{\beta+1}}$. Let  $x\in B_1$ and $y\in B_2$. Since
  $$\mu_{x+y}(t)\geq \mu_x(\frac{t}{2^{\beta+1}})\wedge \mu_y(\frac{t}{2^{\beta+1}})>1-\alpha,$$
  we get $x+y\in B=B(0, \alpha, t)$. Hence the vector addition is continuous. Now,  for $B=B(0, \alpha, t)$
  and $\lambda\in \mathbb{R}$, by choosing $\alpha_1<\alpha$, $t_1<\frac{t}{2|\lambda|^\beta}$
  and $r=(\frac{t}{2t_1})^{\frac{1}{\beta}}$ we see that  $B_1=B(0, \alpha_1, t_1)$
   and $N_r(\lambda):=\{\xi: |\xi-\lambda|<r\}$ with $r>0$ are such that $N_r(\lambda) B_1\subset B$. In fact, if $x\in B_1$ and
   $\xi\in N_r(\lambda)$, then  \begin{center}
 \begin{tabular}{lllllllll}$\mu_{\xi x}(t)$&$=$&$\mu_{(\xi x-\lambda x)+\lambda x}(t)$\\
 &$\geq$&$\mu_{(\xi-\lambda)x}(\frac{t}{2})\wedge\mu_{\lambda x}(\frac{t}{2})$\\
 &$=$&$\mu_x(\frac{t}{2|\xi-\lambda|^\beta})\wedge\mu_x(\frac{t}{2|\lambda|^\beta})$\\&$>$&$1-\alpha,$ \end{tabular}
 \end{center}
 that is,  $\xi x\in B=B(0, \alpha, t)$.
  Therefore $(X,\mu)$ is a topological vector space.
\end{proof}
\begin{remark} {\rm If $(X,\mu)$ is  a $\beta$-homogeneous $\mathcal{PM}$-space  satisfying $(\mathbf{\Upsilon})$,
then $(X, \mu)$ is first-countable.  In fact, by Lemma \ref{11}, it is
enough to choose the family $\{B(0, \alpha, t): \alpha\in
\mathbb{Q}\cap (0,1), t>0\}$.}
\end{remark}
\begin{definition}\rm{  Let $(X,\mu)$ be a $\mathcal{PM}$-space.
\begin{enumerate}
\item  A set $ B\subseteq X$ is said to be balanced if $\lambda
B\subseteq B$ for
 every $\lambda\in\mathbb{R}$ with $|\lambda|\leq 1$.
\item   A set $C\subseteq X$ is said to be  convex if $\lambda
C+(1-\lambda)C\subseteq C$
 for  every$\lambda\in \mathbb{R}$ with $ \lambda\in (0,1)$.
\end{enumerate}}
\end{definition}
\begin{lm} Let $(X, \mu)$ be a $\beta$-homogeneous $\mathcal{PM}$-space. Then $B(0, \alpha, t)$ is a convex  balanced set, for
every $\alpha \in (0,1)$ and $t>0$.
\end{lm}
\begin{proof}
 If $\lambda=0$, then $0\cdot B(0, \alpha,
t)=\{0\}\subseteq B(0, \alpha, t)$. If $\lambda\neq 0$, then
$$\lambda\cdot B(0, \alpha, t)=\{\lambda y: \mu_{y}(t)>1-\alpha\}\
=\ \{x: \mu_x(|\lambda|^\beta t)>1-\alpha\}.$$
 Let $|\lambda|\leq 1$. Then $|\lambda|^\beta t<t$, and therefore  $\mu_x(t)\geq \ \mu_x(|\lambda|^\beta t)>1-\alpha$, i.e.,
  $x\in B(0, \alpha, t)$. If  $x, y\in B(0, \alpha, t)$ and $\lambda\in (0,1)$,  then  $\mu_x(t)> 1-\alpha$ and
 $\mu_y(t)>1-\alpha$ imply that  $$\mu_y(\frac{t}{|1-\lambda|^\beta})>1-\alpha\,\,\,\, \mbox{and} \,\,\,\,
 \mu_x(\frac{t}{|\lambda|^\beta})>1-\alpha.$$ This implies that $$\mu_{\lambda x+(1-\lambda)y}(t)\geq \mu_x(\frac{t}{2|\lambda|^\beta}) \wedge \mu_y(\frac{t}{2|1-\lambda|^\beta})>1-\alpha,$$ and therefore  $$\mu_{\lambda x+(1-\lambda)y}(t)>1-\alpha.$$
 Hence $\lambda x+(1-\lambda)y\in B(0, \alpha, t)$.
\end{proof}
%

\begin{theorem} Let $(X,\mu)$ be a $\mathcal{PM}$-space and $\tau_\mu$ be the topology induced by $\mu$.
Then for every sequence $\{x_n\}$ in $X$,  $x_n \rightarrow x$ in
topology
 $\tau_\mu$, if and only if \ $\lim_{n \rightarrow +\infty} \mu_{x_n-x}(t) =1$, for all $t>0$.
\end{theorem}
\begin{proof} Fix a $t>0$. If $x_n \rightarrow x$ in topology $\tau_\mu$, then for every
 $\epsilon\in (0,1)$, there exists $n_0\in \mathbb{N}$ such that for
  every $n\geq n_0$, $x_n\in B(x, \epsilon, t)$. It means
  $\mu_{x_n-x}(t)>1-\epsilon$, for all $n\geq n_0$. Thus
  $1-\mu_{x_n-x}(t)<\epsilon$, for all $n\geq n_0$. Hence
  $\lim_{n \rightarrow +\infty} \mu_{x_n-x}(t) =1$, for all $t>0$.
Conversely, for $t>0$, let $\lim_{n \rightarrow +\infty}
\mu_{x_n-x}(t) =1$. Then for every
 $\epsilon\in (0,1)$ there exists $n_0\in \mathbb{N}$ such that
  $1-\mu_{x_n-x}(t)<\epsilon$, for every $n\geq n_0$, i.e.,  $x_n\in B(x, \epsilon, t)$ for every $n\geq n_0$. Hence $x_n \rightarrow x$ in topology $\tau_\mu$.
\end{proof}


\end{document}